\documentclass[a4paper,
leqno,11pt
]{amsart}
\usepackage[utf8]{inputenc}
\usepackage[T1]{fontenc}
\usepackage{hyperref}
\usepackage{microtype}
\usepackage{booktabs}
\usepackage{enumitem}
\usepackage{multicol}
\usepackage{color}
\usepackage{amsmath, amsthm, amssymb}
\usepackage{graphicx}
\usepackage{caption}
\newcommand{\K}{\ensuremath{\mathbb{K}}}
\newcommand{\C}{\ensuremath{\mathbb{C}}}
\newcommand{\R}{\ensuremath{\mathbb{R}}}
\newcommand{\N}{\ensuremath{\mathbb{N}}}
\newcommand{\Q}{\ensuremath{\mathbb{Q}}}
\newcommand{\Z}{\ensuremath{\mathbb{Z}}}
\newcommand{\CP}{\ensuremath{\mathbb{CP}}}
\newcommand{\HP}{\ensuremath{\mathbb{HP}}}
\newcommand{\coh}{\ensuremath{\mathrm{H}}}

\newcommand{\SU}{\ensuremath{\mathrm{SU}}}
\newcommand{\T}{\ensuremath{\mathrm{T}}}
\newcommand{\Sph}{\ensuremath{\mathrm{S}}}
\newcommand{\betti}{\ensuremath{\mathrm{b}}}

\newcommand{\isom}{\cong}
\newcommand{\tensor}{\ensuremath{\otimes}}

\DeclareMathOperator{\Kern}{ker}
\DeclareMathOperator{\Image}{im}
\DeclareMathOperator{\rank}{rk}

\newcommand{\avoidbreak}{\postdisplaypenalty=1000}

\newtheoremstyle{thmstyle}
{0.6cm}
{0.6cm}
{\itshape}
{}
{\bfseries}
{.}
{0.5em}
{}

\newtheoremstyle{remstyle}
{0.6cm}
{0.6cm}
{}
{}
{\bfseries}
{.}
{0.5em}
{}

\theoremstyle{thmstyle}
\newtheorem{thm}{Theorem}[section]
\newtheorem{prop}[thm]{Proposition}
\newtheorem{cor}[thm]{Corollary}
\newtheorem{fact}[thm]{Fact}
\newtheorem{lem}[thm]{Lemma}

\newtheorem*{defi*}{Definition}
\theoremstyle{remstyle}
\newtheorem{rem}[thm]{Remark}

\newcommand{\mytableextraspace}{\addlinespace[.3em]}

\title[rationally elliptic manifolds in low dimensions]{Classification and characterization of\\ rationally elliptic manifolds in low dimensions}
\author{Martin Herrmann}
\address{Martin Herrmann \\Fakultät für Mathematik \\Karlsruher Institut für Technologie \\Kaiserstraße 89--93 \\76133 Karlsruhe, Germany}
\email{martin.herrmann@kit.edu}
\begin{document}
\begin{abstract}
We give a characterization of closed, simply connected, rationally elliptic 6--manifolds in terms of their rational cohomology rings and a partial classification of their real cohomology rings. We classify rational, real and complex homotopy types of closed, simply connected, rationally elliptic 7--manifolds. We give partial results in dimensions 8 and 9.
\end{abstract}

\maketitle

\section{Introduction}

A closed, simply connected manifold $M$ is called rationally elliptic if 
\[\dim\pi_*(M)\tensor \Q=\sum_{k\geq 2}\dim \pi_k(M)\tensor \Q< \infty.\]
For a simply connected  space $X$ we additionally require that the rational cohomology of  $X$ satisfies $\sum_{k\geq0}\dim\coh^k(X;\Q)<\infty$. The definition can be generalized to nilpotent spaces.
 
The importance of rationally elliptic manifolds for Riemannian geometry mainly stems from the conjecture, attributed to Bott, that a closed, simply connected manifold of (almost) nonnegative sectional curvature is rationally elliptic (see \cite{GH82}).

A positive answer to this conjecture would, for example, imply Gromov's conjecture that the bound for the sum of the Betti numbers of a nonnegatively curved $n$-manifold is bounded by $2^n$, see \cite{FrH} and \cite{Pavlov} for an improved estimate for simply connected spaces.

Rationally elliptic spaces have some nice properties. For example, by the work of Halperin \cite{Halperin77}  the rational cohomology ring $\coh^*(X;\Q)$ of a rationally elliptic space $X$  satisfies Poincaré duality  and the sequence of the Betti numbers of the loop space $\Omega X$ grows polynomially, i.e. $\sum_{i=0}^k\betti_k(\Omega X)\leq k^m$ for some integer $m$, while for a rationally hyperbolic space it grows exponentially (see \cite[Proposition 33.9]{FHT}).

Examples of rationally elliptic manifolds include homogeneous spaces and biquotients of compact Lie groups (by a theorem of Hopf) and co\-homo\-geneity one manifolds (see \cite{GroveHalperin}). Furthermore, if $F\to E \to B$ is a fibre bundle where $E$, $F$ and $B$ are manifolds, then if two of these spaces are rationally elliptic and the third is nilpotent, then the third space is rationally elliptic by the associated exact homotopy sequence.

The classification of closed, simply connected, rationally elliptic manifolds of dimension five or less is  known:
\begin{fact}
A closed, simply connected, rationally elliptic manifold of dimension five or less is
\begin{itemize}
\item diffeomorphic to $\Sph^2$ or $\Sph^3$,
\item homeomorphic to $\Sph^4$, $\Sph^2 \times \Sph^2$, $\CP^2$, $\CP^2 \# \CP^2$ or $\CP^2 \# \overline{\CP}^2$, or
\item rationally homotopy equivalent to $\Sph^5$ or $\Sph^2\times \Sph^3$.
\end{itemize}
\end{fact}
For the 4--dimensional case see \cite[Lemma 3.2]{PatPet03}. The 5--dimensional case follows easily from the classification of possible exponents in this dimension, which is easily done with the results of Section \ref{susubsec:Exponents}. Note that there are infinitely many integral homotopy types of closed, simply connected, rationally elliptic 5--manifolds, which can be seen from Barden's classification of closed, simply connected 5--manifolds in \cite{Barden}.

Our first theorem gives a characterization of closed, simply connected, rationally elliptic 6--manifolds in terms of their cohomology rings.
\begin{thm}\label{TheoremDimension6Rational}
A closed, simply connected 6--manifold $M$ is rationally elliptic if and only if one of the following holds
\begin{enumerate}[label={\rm(\alph*)}]
\item $\betti_2(M)=\betti_3(M)=0$;\label{b2=0b3=0}
\item $\betti_2(M)=0$ and $\betti_3(M)=2$;\label{b2=0b3=2}
\item $\betti_2(M)=1$ and $\betti_3(M)=0$;\label{b2=1b3=0}
\item $\betti_2(M)=2$, $\betti_3(M)=0$ and $\coh^*(M;\Q)$ is generated by $\coh^2(M;\Q)$;\label{b2=2b3=0}
\item $\betti_2(M)=3$, $\betti_3(M)=0$, $\coh^*(M;\Q)$ is generated by $\coh^2(M;\Q)$ and there is a basis $x_1,x_2,x_3$ of $\coh^2(M;\Q)$, such that the kernel of the restriction of the homomorphism  $\Q[\tilde{x}_1, \tilde{x}_2, \tilde{x}_3] \to \coh^*(M;\Q)$ with $\tilde{x}_i \mapsto x_i$ to homogeneous polynomials of degree two has a regular sequence as a basis.\label{b2=3b3=0}
\end{enumerate}
\end{thm}
Note that, in dimension up to six, every closed, simply connected manifold is formal by a theorem of Miller (see \cite{Miller79}), so a classification of rational (or real) cohomology rings is equivalent to a classification of rational (or real) homotopy types. The rational (respectively real) cohomology rings of these manifolds are determined by their third Betti number and a cubic form on the second cohomology group with rational (respectively real) coefficients. In the real case we can give a classification of the real homotopy types for closed, simply connected, rationally elliptic 6--manifolds $M$ with second Betti number  $\betti_2(M)\leq 2$.

\begin{thm}\label{TheoremDimension6Realb2leq2}
A closed, simply connected, rationally elliptic 6--manifold $M$  with $\betti_2(M)\leq 2$ has the real homotopy type of exactly one of the following manifolds:
\[\Sph^6, \Sph^3\times\Sph^3, \CP^3, \Sph^2 \times \Sph^4, \CP^2\times\Sph^2,  \SU(3)/ \T^2\text{ or }\CP^3 \# \CP^3.\]
\end{thm}

In the case $\betti_3(M)=3$ we can give a classification of the possible cubic forms.

\begin{thm}\label{TheoremDimension6Realb2eq3}
A closed, simply connected 6--manifold $M$ with $\betti_2(M)=3$  is rationally elliptic, if and only if $\betti_3(M)=0$ and the cubic form associated to $\coh^*(M;\R)$ is equivalent to $x y z$,  $z(x^2+y^2)$, $z(x^2+y^2-z^2)$, $x(x^2+y^2-z^2)$, $x(x^2+y^2+z^2)$, $x^3+3x^2 z-3 y^2 z$, $x^3-3 x^2 z-3y^2 z$ or $x^3+y^3+z^3+6 \sigma x y z$ for $\sigma\neq 0, 1, -\tfrac{1}{2}$.
\end{thm}
As a by-product of the proof of Theorem~\ref{TheoremDimension6Realb2leq2}  we get a classification of certain rationally hyperbolic 6--manifolds.
\begin{cor}\label{cor:hyperbolicDimension6}
A closed, simply connected 6--manifold $M$ with $\betti_2(M)\leq 2$ and $\betti_3(M)=0$ is rationally hyperbolic if and only if it has the real homotopy type of $(\Sph^2 \times \Sph^4) \#(\Sph^2 \times \Sph^4)$ or $\CP^3 \# (\Sph^2 \times \Sph^4)$.
\end{cor}

A similar statement for the real  cubic forms associated to closed, simply connected, rationally hyperbolic 6--manifolds with $\betti_2=3$ and $\betti_3=0$ can be read off Table~\ref{TableCubicFormsExamples}.

In the seven-dimensional case we can classify the rational homotopy types. Note that the manifolds in the theorem have pairwise distinct rational homotopy types.

\begin{thm}\label{TheoremDimension7Rational}
A closed, simply connected 7--manifold is rationally  elliptic if and only if it has the rational homotopy type of one of the following manifolds: \\
$\Sph^7$, $\Sph^2 \times \Sph^5$, $\CP^2 \times \Sph^3$, $\Sph^3 \times \Sph^4$, $N^7$ or $M_\sigma^7$ for some $\sigma \in \Q^* / (\Q^*)^2$.
\end{thm}
Here the manifolds $M^7_\sigma$ are realizations of certain minimal models which exist by Sullivan's realization result (see Section~\ref{subsubsec:RealizationManifold}). We can choose \[M^7_{[1]}=\Sph^3 \times (\CP^2 \#  \CP^2)\text{ and }M^7_{[-1]}=\Sph^3 \times (\CP^2 \#  \overline{\CP}^2),\] where $\overline{\CP}^2$ denotes $\CP^2$ with reversed orientation. For $\sigma\neq [\pm1]$ we do not know of a nice realization of $M_\sigma$ as a manifold (see Proposition~\ref{PropM7sigmanot}), but $M_\sigma$ is rationally homotopy equivalent to a nonnegatively curved orbifold (see Remark~\ref{RemarkDefinitionXsigma}). The manifold $N^7$ is a homogeneous space $(\SU(2))^3/\,\T^2$.  Furthermore $N^7$ is an example of a non-formal manifold (see \cite[Example 2.91]{FOT}).
 
This paper is organized as follows. In Section~\ref{sec:Prelim} we recall some preliminaries on rational homotopy theory and the cohomology rings of 6--manifolds. Section~\ref{sec:DimensionSix} is divided into two parts in which Theorems~\ref{TheoremDimension6Rational},~\ref{TheoremDimension6Realb2leq2} and~\ref{TheoremDimension6Realb2eq3} are proven. In Section~\ref{sec:DimensionSeven}, we prove Theorem~\ref{TheoremDimension7Rational} and we also give a classification of the real and complex homotopy types of closed, simply connected, rationally elliptic 7--manifolds. In Section~\ref{sec:DimensionsEightAndNine} we state and prove some partial classification results in dimensions 8 and 9.

The results in this article were part of the author's dissertation \cite{MHDiss} at the Karlsruhe Institute of Technology. Part of the research was carried out at the University of Fribourg. The author wishes to thank his advisor Wilderich Tuschmann and Anand Dessai for helpful and stimulating discussions. Furthermore, he wishes to thank the referee of a previous version of this manuscript for various helpful remarks.

\section{Preliminaries}\label{sec:Prelim}
\subsection{Rational homotopy theory}
For rational homotopy theory, we use the books \cite{FHT} and \cite{FOT} as  references and use their notation. For the convenience of the reader, we give an overview over the results that we need.

\subsubsection{Basic definitions}
Let $\K$ be a field of characteristic zero. By $X$ we always denote a simply connected space $X$ with finite Betti numbers. A \emph{commutative differential graded algebra}  (cdga henceforth) $(A,d)$ over $\K$ is a graded algebra $A=\bigoplus_{k\geq0} A^k$ with unit which is commutative in the graded sense, that is $ab=(-1)^{pq} b a$ for $a\in A^p$, $b\in A^q$, together with a linear differential $d:A\to A$ satisfying $d^2=0$, $d(A^k)\subset A^{k+1}$ and $d(ab)=d(a) \;b+(-1)^p a\; d(b)$ for $a\in A^p$.

 For a graded vector space $V=\bigoplus_{k\geq 0} V^k$, we denote by $\Lambda V$ the tensor product of the polynomial algebra on $V^\text{even}=\bigoplus_{k\geq 0} V^{2k}$ and the outer algebra on $V^\text{odd}=\bigoplus_{k\geq0}V^{2k+1}$. If $x_1,\dots, x_n$ is a (homogeneous) basis of $V$  we also write $\Lambda(x_1,\dots,x_n)$ for $\Lambda V$. Furthermore, we will use the following conventions.  The elements of degree $k$ in the graded algebra $\Lambda V$ will be denoted by $(\Lambda V)^k$, while we denote by $\Lambda^k V$ the linear subspace generated by elements of word length $k$ in $V$. Furthermore $\Lambda V^k=\Lambda(V^k)$. The degree of a homogeneous element $v\in \Lambda V$ will be denoted by $|v|$.

A \emph{Sullivan algebra} is a cdga $(\Lambda V,d)$ with $V=V^{\geq1}$ such that there exists a basis $\{x_\alpha\}_{\alpha\in I}$ with $I$ a well-ordered index set, such that $d x_i \in \Lambda(x_j, j<i)$. If $V^1=\{0\}$ the existence of such a basis follows for every $(\Lambda V,d)$.

A Sullivan algebra  $(\Lambda V,d)$ is called \emph{minimal} if  $d(V)\subset \Lambda^{\geq2}V$.

If $(A,d)$ is a cdga with $\coh^0(A,d)\isom \K$, then there exists a \emph{minimal model} of $(A,d)$, that is a minimal Sullivan algebra $(\Lambda V,d)$ and a homomorphism $\varphi:(\Lambda V,d) \to (A,d)$ inducing an isomorphism in cohomology. The minimal model is unique up to isomorphism.

To a space $X$ one can associate a cdga $(\mathrm{A_{PL}}(X;\K),d)$ (see \cite[Chapter 10]{FHT}), such that  $\coh^*(X;\K)\isom \coh^*(\mathrm{A_{PL}}(X;\K),d)$. The $\K$--minimal model of $X$ is the minimal model of $(\mathrm{A_{PL}}(X;\K),d)$.

If $(\Lambda V,d)$ is the rational minimal model of $X$, then $(\Lambda V,d)\tensor \K$ is the $\K$-minimal model of $X$. We say that $X$ and $Y$ have the same \emph{$\K$--homotopy type}, if their $\K$--minimal models are isomorphic, and write $X\simeq_\K Y$. For $\K=\Q$ this is equivalent to the usual definition.

If $(\Lambda V,d)$ is the rational minimal model of a simply connected space $X$, then $V^1=\{0\}$ and $V^k\isom\mathrm{Hom}(\pi_k(X),\Q)$. A  minimal Sullivan algebra $(\Lambda V,d)$ is called \emph{rationally elliptic} if $\dim V=\sum_{k}\dim V^k< \infty$ and $\dim \coh^*(\Lambda V,d)< \infty$.

\subsubsection{Realization of minimal models by manifolds}\label{subsubsec:RealizationManifold}
For a cdga $(A,d)$ the \emph{formal dimension} is defined as the maximal $k\in \N$ with $\coh^k(A,d)\neq\{0\}$, if such a $k$ exists, else it is defined to be $\infty$.

By a theorem of Sullivan \cite[Section 13]{Sull}, compare also \cite{Barge76} and \cite[Theorem 3.2]{FOT}, the following holds: 

Let $(\Lambda V, d)$ be a rational  minimal Sullivan algebra of formal dimension $n$ with $V=V^{\geq2}$ and let $\coh^*(\Lambda V,d)$ satisfy Poincaré duality. Then, if $n$ is not divisible by 4, there is a compact simply connected manifold realizing $(\Lambda V,d)$. If $n=4k$ is divisible by 4 and the signature of the quadratic form on $\coh^{2k}(\Lambda V,d)$ is zero, then $(\Lambda V,d)$ is realizable by a compact, simply connected manifold, if and only if in some basis of $\coh^{2k}(\Lambda V,d)$ and for some identification $\coh^{4k}(\Lambda V,d)\isom \Q$ the form is given by $\sum\pm x_i^2$. In the case that the signature is nonzero, there are additional conditions on chosen Pontryagin numbers. Here, for $n=4k$, we will only use the case, where the signature is zero.

By a theorem of Halperin \cite[Theorem 3]{Halperin77} a rationally elliptic minimal model satisfies Poincaré duality. Therefore, every simply connected, rationally elliptic minimal Sullivan algebra of formal dimension $n$, with $n$ not divisible by 4, is the minimal model of a compact, simply  connected $n$-manifold.

\subsubsection{Exponents}\label{susubsec:Exponents}
Recall that the (a- and b-)exponents of a rationally elliptic, minimal Sullivan algebra $(\Lambda V,d)$ are $a \in \N^q$ and $b \in \N^r$ if there exist homogeneous bases $x_1,\dots,x_q$ of $V^{\mathrm{even}}$ and $y_1, \dots,y_r$ of $V^{\mathrm{odd}}$, such that $|x_i|=2 a_i$  and $|y_j|=2 b_j-1$. The pairs of tuples $a \in \N^q$ and $b \in \N^r$ that arise as exponents of rationally elliptic minimal Sullivan algebras have a purely arithmetic description.

\begin{defi*}[Strong arithmetic condition (SAC)]
The tuples $a \in \N^q$ and $b \in \N^r$ satisfy (SAC) if for all $1\leq s\leq q$ and $1\leq i_1<\dots<i_s\leq q$ there exist $1\leq j_1<\dots<j_s\leq r$ such that there are $\gamma_{kl}\in \N_0$ with 
\[b_{j_k}=\sum_{l=1}^s \gamma_{kl} a_{i_l}\qquad \text{and} \qquad \sum _{l=1}^s\gamma_{kl}\geq 2\avoidbreak\] 
for all $k=1,\dots,s$.
\end{defi*}

Friedlander and Halperin showed in \cite{FrH} that $a\in \N^q$ and $b \in \N^r$ with $b_j\geq 2$ for $j=1,\dots, r$ arise as the exponents of a simply connected, rationally elliptic minimal Sullivan algebra if and only if they satisfy (SAC). Furthermore the exponents of a simply connected, rationally elliptic minimal Sullivan algebra $(\Lambda V,d)$ satisfy (see \cite{FHT})
\begin{enumerate}[label=\textnormal{(\alph*)}]
\item $\dim V^{\text{even}}=q\leq r =\dim V^{\text{odd}} $\textup{;}
\item $ \sum_{i=1}^q 2a_i\leq n$\textup{;}
\item $\sum_{j=1}^r (2b_j-1) \leq 2n-1$\textup{;}
\item $n= 2 \left(\sum_{j=1}^r b_j - \sum_{i=1}^q a_i \right) -( r-q) $,
\end{enumerate}
where $n$ is the formal dimension of $(\Lambda V,d)$.

This is enough to compute the possible vector spaces $V$ that arise in the minimal models $(\Lambda V,d)$ of closed, simply connected manifolds of a given dimension.   

\subsubsection{Pure Sullivan algebras and regular sequences}\label{subsubsec:PureSullivanAlgebrasRegularSequences}The notions of pure Sullivan Algebras and regular sequences will be essential in the proof of our results in dimension~6. 

A Sullivan algebra $(\Lambda V,d)$ is called pure if $\dim V<\infty$, $d(V^{\mathrm{even}})=0$ and $d(V^{\mathrm{odd}})\subset\Lambda V^{\mathrm{even}}$.

Let $R$ be a ring. Recall that a sequence $r_1,r_2,\dots,r_k$ of elements of $R$ is called \emph{regular} if $r_1$ is not a zero divisor in $R$ and $r_i$ is not a zero divisor in $R/(r_1,\dots,r_{i-1})$ for $i=2,\dots,n$. In general being regular depends on the order of the sequence $r_1,\dots,r_n$. However, we are only interested in the case where $R=\K[x_1,\dots,x_n]$ is a polyomial ring over a field and the $r_i$ are homogeneous polynomials. In this case, being regular does not depend on the order of the elements $r_1, \dots,r_n$ (see \cite[Corollary to Theorem 16.3]{Matsumura} for example).

These two notions can be brought together as follows. Let $(\Lambda V,d)$ be a pure minimal Sullivan algebra with $\dim V^{\mathrm{even}}=\dim V^{\mathrm{odd}}$ and $y_1, \dots,y_k$ a basis of $V^{\mathrm{odd}}$, then $(\Lambda V,d)$ is rationally elliptic if and only if $dy_1,\dots,dy_k$ is a regular sequence. Furthermore, if $(\Lambda V,d)$ is rationally elliptic, then $\coh^*(\Lambda V,d)\isom \Lambda V^{\mathrm{even}}/(dy_1,\dots,dy_k)$. This follows from \cite[Propositions 32.2 and 32.3]{FHT} and \cite[Corollary 3.2]{Stanley78}.

\subsection{Cohomology rings of 6--manifolds}
Let $\K$ be a field of characteristic zero. By a result of Miller \cite{Miller79}, in dimensions $\leq 6$ every closed, simply connected manifold $M$ is formal, i.e. its minimal model over $\K$ is also a minimal model for the cdga $(\coh^*(M;\K),0)$. Due to the uniqueness of the minimal model, two formal spaces have the same $\K$--homotopy type if and only if their cohomology rings with coefficients in $\K$ are isomorphic. Therefore, in dimension 6 we only need to consider the cohomology rings.

The isomorphism class of the cohomology ring $\coh^*(M;\K)$ of a closed, simply connected 6--manifold $M$ is determined by the dimension of $\coh^3(M;\K)$ and the equivalence class of the cubic form on $\coh^2(M;\K)$ given by the cup product to $\coh^6(M;\K)\isom \K$. The equivalence relation we use is given by changing the basis of $\coh^2(M;\K)$ and scaling the form by a number in $\K$ (the scaling isn't necessary for $\K=\R$ or $\C$).

By a result of Wall \cite{Wall}, every rational cubic form is also realizable as the form associated to a closed, simply connected, spin manifold of dimension 6 with $\betti_3=0$ and torsion free homology.

We will use two equivalent definitions of cubic forms on a vector space $V$ of finite dimension $n$ in this paper. The first is that of a symmetric multilinear map \[F: V \times V \times V \to \K,\] which is uniquely determined by the coefficients $F_{ijk}=F(e_i,e_j,e_k)$ with $i\leq j\leq k$ for some basis $e_1,\dots,e_n$ of V. The second description is that of a homogeneous polynomial of degree 3 in $n$ variables. 

These definitions can be identified via
\[F\mapsto F(\sum_{i=1}^n x_i e_i,\sum_{i=1}^n x_i e_i,\sum_{i=1}^n x_i e_i)\in \K[x_1,\dots,x_n].\]

\section{Six-dimensional manifolds}\label{sec:DimensionSix}
\subsection{The rational case (proof of Theorem~\ref{TheoremDimension6Rational})}
The possible exponents have already been calculated by Pavlov using the results of Friedlander and Halperin mentioned in Section \ref{susubsec:Exponents}. 
\begin{lem}[See \cite{Pavlov}]\label{Lemma6dimensionalExponents}
A closed, simply connected, rationally elliptic 6--manifold has one of the following exponents:
\begin{multicols}{2}
\begin{enumerate}
\item[(6.1)] $a=(~)$, $b=(2,2)$
\item[(6.2)] $a=(1)$, $b=(4)$
\item[(6.3)] $a=(3)$, $b=(6)$
\item[(6.4)] $a=(1,1)$, $b=(2,3)$
\item[(6.5)] $a=(1,2)$, $b=(2,4)$
\item[(6.6)] $a=(1,1,1)$, $b=(2,2,2)$
\end{enumerate}
\end{multicols}
\end{lem}
In four of these cases the minimal model is already determined by its vector space structure.
\begin{lem}[See \cite{Pavlov}]
A closed, simply connected, rationally elliptic 6--manifold with exponents like in
\begin{itemize}
\item (6.1) is rationally homotopy equivalent to $\Sph^3\times \Sph^3$;
\item (6.2) is rationally homotopy equivalent to $\CP^3$;
\item (6.3) is rationally homotopy equivalent to $\Sph^6$;
\item (6.5) is rationally homotopy equivalent to $\Sph^2 \times \Sph^4$,
\end{itemize}
\end{lem}
We will now deal with case (6.4). Let  $(\Lambda \tilde{V},d)=(\Lambda(x_1,x_2,y_1,y_2),d)$ be given by $|x_i|=2, |y_1|=3, |y_2|=5$ and $ dx_i=0$, $dy_1=x_1^2+f_2 \, x_2^2$, and $ dy_2 = g_1 \,x_1^3 + g_2 \, x_1^2 x_2 + g_3 \, x_1 x_2^2 + g_4 \, x_2^3$ for some $f_2,g_1,g_2,g_3,g_4 \in \Q$. 

Note that the minimal model of a closed, simply connected, rationally elliptic 6--manifold with exponents like in (6.4) is of this form: The quadratic form given by $dy_1$ cannot vanish, so one can choose an orthogonal basis for it and rescale.

\begin{lem}\label{LemmaVTildeModels}
The above model $(\Lambda \tilde{V},d)$ is the minimal model of a closed, simply connected 6--manifold if and only if 
\[ g_4\neq f_2 g_2\pm \sqrt{-f_2} (f_2 g_1-g_3)\tag{$*$}.\]
\end{lem}
\begin{proof}
To see that $(*)$ is necessary, one can compute the determinant of the differential  $d_7:  (\Lambda \tilde{V})^7 \to \Kern d_8$ in the bases  $y_1x_1^2, y_1x_1x_2, y_1x_2^2, y_2x_1,y_2x_2$ of $(\Lambda \tilde{V})^7$ and $x_1^4, x_1^3 x_2,x_1^2x_2^2x_1x_2^3,x_2^4$ of $\Kern d_8$. It is $f_2^3 g_1^2 + f_2^2 g_2^2 - 2 f_2^2 g_1 g_3 + f_2 g_3^2 - 2 f_2 g_2 g_4 + g_4^2\neq 0$. Solving for $g_4$, this gives $(*)$.

To see that $(*)$ is sufficient we only need to prove that $\coh^*(\Lambda \tilde{V},d)$ is finite dimensional. If we have done so, the formal dimension needs to be 6 due to its exponents and by the results mentioned in Section \ref{subsubsec:RealizationManifold} it is realized by a compact, simply connected 6--manifold. We show that $\dim \coh^{\geq9}(\Lambda \tilde{V},d)=0$  by an elementary calculation. 

Let $k\geq4$. It is easy to see that $d_{2k}$ is injective when restricted to the span  of $y_1y_2 x_1^ix_2^{k-4-i}$, $i=0,\dots,k-4$. So $\dim( \Image d_{2k})=k-3$ and $\dim( \Kern d_{2k})=k+1$.

The image of $d_{2k+1}$ is generated by 
\[v_i=d(y_1x_1^{k-i} x_2^{i-1})= x_1^{k-i+2}x_2^{i-1}+f_2 x_1^{k-i} x_2^{i+1}, \quad i=1, \dots,k\]
and 
\begin{align*} w_j&=d(y_2x_1^{k-1-j} x_2^{j-1})\\&=g_1 \, x_1^{k+2-j} x_2^{j-1}+g_2 \, x_1^{k+1-j} x_2^{j}+ g_3 \, x_1^{k-j} x_2^{j+1}+g_4 \, x_1^{k-1-j} x_2^{j+2}\end{align*}
for $j=1,\dots, k-1$. 

Let
\begin{align*}u_1&=w_{k-2}-g_1 \, v_{k-2} - g_2 \, v_{k-1} - (g_3-f_2 g_1) v_k\\&= (g_4-f_2 g_2 )\; x_1x_2^{k}- f_2 (g_3 - f_2 g_1)\, x_2^{k+1}  \end{align*}
and 
\begin{align*}u_2&=w_{k-1}-g_1 \, v_{k-1} - g_2 \, v_{k} \\&= (g_3-f_2 g_1 )\; x_1x_2^{k}+ (g_4 - f_2 g_2)\, x_2^{k+1}.  \end{align*}

Because of $(*)$, the elements $v_1,\dots, v_k,u_1,u_2$ are linearly independent. So $\dim \Image d_{2k+1}\geq k+2 =\dim \ker d_{2k+2}$ and therefore $\Image d_{2k+1}=\Kern d_{2k+2}$. By also computing their dimensions, we get $\Image d_{2k}=\Kern d_{2k+1}$.
\end{proof}

\begin{rem}
It is easy to see that the equivalence class of $f_2$ in $\Q/(\Q^*)^2$ is an invariant of the isomorphism class of $(\Lambda \tilde{V},d)$. Since for every $f_2 \in \Q$, one can choose $g_1,\dots,g_4$ such that $(*)$ holds, there are infinitely many rational homotopy types of closed, simply connected, rationally elliptic 6-manifolds with $\betti_2=2$, in contrast to real homotopy types of these. 
\end{rem}

In the following we assume that $(\Lambda \tilde{V},d)$ satisfies $(*)$.

Let $\omega_1$ and  $\omega_2$ be the cohomology classes of $x_1$ and $x_2$ and  $\alpha_1=f_2 g_1-g_3$ and $\alpha_2=f_2g_2-g_4$.
Then $\omega_1^3= -f_2 \omega_1 \omega_2^2$ and $\omega_1^2 \omega_2=-f_2 \omega_2^3$. Therefore
\[0=g_1\omega_1^3 + g_2 \omega_1^2 \omega_2 + g_3\omega_1 \omega_2^2+g_4\omega_2^3=-(\alpha_1\; \omega_1 \omega_2^2 + \alpha_2\; \omega_2^3).\]
Then $\Omega=-\alpha_2\; \omega_1 \omega_2^2+\alpha_1\; \omega_2^3\neq 0$, since $(\alpha_1, \alpha_2)\neq(0,0)$ due to $(*)$. We have $(\alpha_1^2 +\alpha_2^2)\omega_1 \omega_2^2=-\alpha_2 \Omega$ and  $(\alpha_1^2 +\alpha_2^2) \omega_2^3=\alpha_1 \Omega$.

Since we can use $\tfrac{1}{\alpha_1^2+\alpha_2^2} \Omega$ to define the cubic form $F$ associated to $\coh^*(\Lambda \tilde{V},d)$, it is given by the components \[F_{111}=f_2 \alpha_2, \quad F_{112} =-f_2 \alpha_1, \quad F_{122}=-\alpha_2, \quad F_{222}=\alpha_1\]
and because of $(*)$, we have $\alpha_2 \neq \pm \sqrt{-f_2} \alpha_1$.

On the other hand, every cubic form that is of this form with given parameters $f_2, \alpha_1,\alpha_2 \in \Q$ satisfying $\alpha_2 \neq \pm \sqrt{-f_2} \alpha_1$ is realized by a minimal model $(\Lambda \tilde{V},d)$ of a closed, simply connected, rationally elliptic 6--manifold.

\begin{lem}
Let $F$ be a cubic form on a two-dimensional vector space $V$ over $\Q$. Then there is a basis of $V$ and  $f_2, \alpha_1,\alpha_2\in \Q$ such that the components of $F$ in this basis are given by
\[F_{111}=f_2 \alpha_2, \quad F_{112} =-f_2 \alpha_1, \quad F_{122}=-\alpha_2, \quad F_{222}=\alpha_1.\]
\end{lem}
\begin{proof}
First we prove that it is possible to find a basis such that $F_{111}F_{222}=F_{112}F_{122}$. 
The change of basis  $\tilde{x}_1=x_1$, $\tilde{x}_2=\lambda x_1 + x_2$ gives
\[\tilde{F}_{111} \tilde{F}_{222} - \tilde{F}_{112} \tilde{F}_{122} = F_{111} F_{222} - F_{112} F_{122}  +\lambda \;2( F_{111} F_{122} - F_{112}^2),\]
where the $\tilde{F}_{ijk}$ are the components with respect to the new basis. This expression vanishes for some $\lambda \in \Q$ if $F_{112}^2 \neq F_{111} F_{122}$. If $F_{112}^2 = F_{111} F_{122}$, then  changing the basis to $\tilde{x}_1=x_1 + \lambda x_2$, $\tilde{x}_2=x_2$ gives
\[\tilde{F}_{111} \tilde{F}_{122} - \tilde{F}_{112}^2= (F_{111} F_{222} - F_{112} F_{122}) \lambda  + (F_{112} F_{222} - F_{122}^2) \lambda^2,\]
so we can arrange $\tilde{F}_{111} \tilde{F}_{122} \neq \tilde{F}_{112}^2$ if the basis doesn't already satisfy $F_{111} F_{222} = F_{112} F_{122}$.

Assume now that $F_{111} F_{222} = F_{112} F_{122}$. If $F=0$, choose $\alpha_1=\alpha_2=0$. If $F\neq 0$, we can assume $F_{122}\neq 0$ or $F_{222} \neq 0$. Then let $\alpha_1=F_{222}$, $\alpha_2=-F_{122}$ and $f_2=-\tfrac{F_{112}}{F_{222}}$ or $f_2=-\tfrac{F_{111}}{F_{122}}$, respectively.
\end{proof}

\begin{lem}\label{LemmaDimension6FormsNotRealized}
If a cubic form $F$ on two-dimensional vector space over $\Q$ is not realized by one of the above models $(\Lambda \tilde{V},d)$ satisfying $(*)$ then it is equivalent to the form associated to $(\Sph^2\times \Sph^4)\#(\Sph^2 \times \Sph^4)$ or $(\Sph^2\times \Sph^4) \# \CP^3$.
\end{lem}

\begin{proof}
Under a general change of basis $\tilde{x}_1=a x_1 + b x_2$, $\tilde{x}_2=c x_1 + dx_2$ and only assuming $F_{111} F_{222}=F_{112}F_{122}$:
\[\tilde{F}_{111} \tilde{F}_{222} - \tilde{F}_{112}\tilde{F}_{122}=2 (b c - a d)^2 \big(a c ( F_{111} F_{122}-F_{112}^2 ) +  b d (F_{112} F_{222}-F_{122}^2)\big),\]
where, as before, $\tilde{F}_{ijk}$ denote the components with respect to the new basis. So if 
\[F_{111} F_{222}=F_{112}F_{122},\quad F_{112}^2 =F_{122} F_{111}\;\text{ and }\;F_{122}^2= F_{112} F_{222}\]
holds in one basis, it holds in every basis.

By the last lemma and the discussion preceding it, we can assume that a cubic form $F$, which is not realized by one of the above models $(\Lambda \tilde{V},d)$ with $(*)$, satisfies $F_{111}=f_2 \alpha_2$, $F_{112} =-f_2 \alpha_1$,  $F_{122}=-\alpha_2$, $F_{222}=\alpha_1$ and $\alpha_2=\pm \sqrt{-f_2} \alpha_1$. 
Therefore
\[F_{111} F_{222}=F_{112}F_{122},\quad F_{112}^2 =F_{122} F_{111}\;\text{ and }\;F_{122}^2= F_{112} F_{222}.\]
If $F\neq0$, we can assume that $F_{222}\neq 0$. Then the change of basis $\tilde{x}_1=x_1+\lambda x_2$, $\tilde{x}_2=x_2$  with  $\lambda=-\frac{F_{122}}{F_{222}}$, gives $\tilde{F}_{122}= F_{122}+\lambda F_{222}=0$, $\tilde{F}_{222}=F_{222}\neq0$ and with the above relations $\tilde{F}_{111}=\tilde{F}_{112}=0$. Scaling to $\tilde{F}_{222}=1$ this is the form associated to $(\Sph^2\times \Sph^4) \# \CP^3$.

If $F=0$ it is the cubic form associated to $(\Sph^2\times \Sph^4)\#(\Sph^2 \times \Sph^4)$.
\end{proof}
The proof of Theorem~\ref{TheoremDimension6Rational} is now easy.

\begin{proof}[Proof of Theorem~\ref{TheoremDimension6Rational}]
By Lemma~\ref{Lemma6dimensionalExponents} the second Betti number of a closed, simply connected, rationally elliptic 6--manifold $M$ satisfies  $\betti_2(M)\leq3$. 
Note that a manifold satisfying (a), (b) or (c) of Theorem~\ref{TheoremDimension6Rational} is rationally homotopy equivalent to $\Sph^6$, $\Sph^3 \times \Sph^3$, $\CP^3$ or $\Sph^2 \times \Sph^4$. 

Now consider a closed, simply connected, rationally elliptic 6--manifold with $\betti_2=2$. By Lemma~\ref{Lemma6dimensionalExponents} and the discussion preceding Lemma~\ref{LemmaVTildeModels} its minimal model is one of the $(\Lambda \tilde{V},d)$ satisfying $(*)$. Therefore it falls into (d) of the theorem. If on the other hand a closed, simply connected 6--manifold $M$ falling into (d) is given, its minimal model has to be one of  $(\Lambda \tilde{V},d)$ satisfying $(*)$ by Lemma~\ref{LemmaDimension6FormsNotRealized}.

Finally, consider a closed, simply connected, rationally elliptic 6--manifold $M$ with $\betti_2(M)=3$, then its rational minimal model has the form $(\Lambda V,d)=(\Lambda(x_1,x_2,x_3,y_1,y_2,y_3),d)$ with $|x_i|=2$, $|y_j|=3$ and $dx_i=0$. In particular, $(\Lambda V,d)$ is a pure Sullivan algebra with an equal number of even and odd generators. As seen in Section~\ref{subsubsec:PureSullivanAlgebrasRegularSequences}, $\coh^*(M;\Q)=\coh^*(\Lambda V,d)=\Lambda(x_1,x_2,x_3)/(dy_1,dy_2,dy_3)$ and $dy_1, dy_2,dy_3$ is a regular sequence. Thus $M$ falls into case (e). 

If, on the other hand, a manifold falling into case (e) is given, then the minimal model has the above form and the manifold is rationally elliptic.
\end{proof}

\subsection{The real case (proof of Theorems~\ref{TheoremDimension6Realb2leq2} and~\ref{TheoremDimension6Realb2eq3}, and Corollary~\ref{cor:hyperbolicDimension6})}
The main difference in approaching the real case is that binary and ternary real cubic forms have been classified in \cite[Lemmas 3 and 4]{McK}. For the rest of this section we will use the definition of a cubic form as a homogeneous polynomial of degree 3 as it is used there. We will now state the classification of McKay \cite{McK}.

A binary real cubic form is equivalent to exactly one of $0$, $x^3$, $x^2y$, $x^3 +y^3$ and $x^2 y - x y^2$. 

A singular ternary real cubic form is equivalent to exactly one of the following:
\begin{multicols}{2}
\begin{itemize} \item$0,$ \item $ x^3,$ \item $ x^2y,$ \item $ x^2 y - x y^2,$ \item $ x(x^2+y^2),$ \item $ x y z,$ \item $ z(x^2+y^2),$ \item $ x(xz -y^2)$ \item$z(x^2+y^2-z^2),$ \item $ x(x^2+y^2-z^2),$ \item $x(x^2+y^2+z^2),$ \item $ x^3-3y^2 z,$ \item $ x^3+3 x^2 z-3 y^2 z$ \item   and $x^3-3 x^2 z-3 y^2 z$.
\end{itemize}
\end{multicols}

A nonsingular ternary real cubic form is equivalent to exactly one of the forms \[x^3+y^3+z^3+6 \sigma \;x y z\] with $\sigma \neq -\tfrac{1}{2}$.

\begin{lem}\label{LemmaBinaryCubicFormsRealization}
The binary real cubic forms are realized by the following manifolds:
\begin{multicols}{2}
\begin{itemize}
\item $0$: $(\Sph^2\times \Sph^4)\#(\Sph^2 \times \Sph^4)$
\item $x^3$: $(\Sph^2\times \Sph^4)\# \CP^3$
\item $x^2y$: $\CP^2 \times \Sph^2$\columnbreak
\item $x^3 +y^3$:  $\CP^3 \# \CP^3 $
\item $x^2 y - x y^2$: $\SU(3)/ \T^2$
\end{itemize}
\end{multicols}
\end{lem}
\begin{proof}
The first four are easy to see. The cohomology ring of  $\SU(3)/\T^2$ has been calculated in \cite{Borel53} and is \[\coh^*(\SU(3)/\T^2;\R)=\Lambda(x_1,x_2)/(x_1^2 +x_1 x_2 + x_2^2, x_1^2 x_2+x_1 x_2^2)\] with $|x_i|=2$. Therefore $x_1^3=x_2^3=0$ and $x_1^2 x_2=-x_1 x_2^2$. So the associated cubic form is as stated. 
\end{proof}

Of these manifolds, $\CP^2 \times \Sph^2$, $\CP^3 \# \CP^3 $ and $\SU(3)/ \T^2$ are rationally elliptic, since they have their rational cohomology ring generated by $\coh^2$. Since the closed, simply connected, rationally elliptic 6--manifolds with second Betti number $\betti_2\leq 1$ have been identified before, this already proves Theorem~\ref{TheoremDimension6Realb2leq2}.

Corollary~\ref{cor:hyperbolicDimension6} also follows from this, since we have seen that every compact, simply connected 6-manifold $M$ with $\betti_2(M)\leq1$ and $\betti_3(M)=0$ is rationally elliptic.

For the proof of Theorem~\ref{TheoremDimension6Realb2eq3}, we start with the following models. For $\lambda\neq1$ let
\[(\Lambda V,d_\lambda) = (\Lambda(x_1,x_2,x_3,y_1,y_2,y_3),d_\lambda)\]
with
$ |x_i|=2$, $|y_j|=3$, $dx_i=0$ and $d_\lambda y_j=x_j^2- \lambda \frac{x_1 x_2 x_3}{x_j}$ for $i,j=1,2,3$.

Let $u_j=x_j^2- \lambda \frac{x_1 x_2 x_3}{x_j}\in\R[x_1,x_2,x_3]$ and suppose there is a $z \in \C^3\setminus\{(0,0,0)\}$ with $u_i(z)=0$ for $i=1,2,3$. Since $z_1^2 =\lambda z_2 z_3$, $z_2^2=\lambda z_1z_3$ and $z_3^2=\lambda z_1 z_2$, we have  $z_i\neq 0$ for $i=1,2,3$. Then $z_1^4=\lambda^2 z_2^2 z_3^2=\lambda ^4 z_1^2 z_2 z_3$, so $\lambda^4 z_2 z_3=z_1^2 = \lambda z_2 z_3$. Therefore $\lambda=1$, which we excluded. So $(0,0,0)$ is the only common zero of $u_1$, $u_2$ and $u_3$ in $\C^3$.  By Hilbert's Nullstellensatz, $\R[x_1,x_2,x_3]/(u_1,u_2,u_3)$ is finite dimensional. By \cite[Propositions 32.1, 32.2 and 32.3]{FHT}, $u_1,u_2,u_3$ is a regular sequence, $(\Lambda V,d_\lambda)$ is rationally elliptic, of formal dimension 6 due to its exponents and its cohomology ring is $\coh^*(\Lambda V,d_\lambda)\isom\R[x_1,x_2,x_3]/(u_1,u_2,u_3)$. 

The cubic form associated to $(\Lambda V,d_\lambda)$ is  $x^3+y^3+z^3+6 \tfrac{1}{\lambda} \;x y z$ if $\lambda \neq 0$ and $xyz$ if $\lambda =0$. So if a closed, simply connected 6--manifold with $\betti_3=0$ has one of these forms associated to it,  it is rationally elliptic. As the models $(\Lambda V,d_\lambda)$ with $\lambda \in \Q\setminus \{1\}$ can obviously be defined over the rational numbers, they can be realized as minimal models of a closed, simply connected 6--manifold and we get the following.
\begin{prop}\label{prop:InfinitelymanyrealHomotopytypesDimension6}
There are infinitely many real homotopy types of closed, simply connected, rationally elliptic 6--manifolds.
\end{prop}

For the remaining cubic forms we can use the same trick. Given a cubic form, we can associate the subspace of the homogenous polynomials of degree 2 in $\R[x_1,x_2,x_3]$ which vanish in the associated cohomology ring $\coh^*(M;\R)$ of some closed, simply connected 6--manifold. To do this, one uses  that such a polynomial $f$ vanishes in the cohomology if and only if $x_1 f$, $x_2 f$ and $x_3f$ vanish in cohomology, which can be seen using the cubic form. If we take for example the cubic form $x^3+y^3+z^3$ (belonging to $\CP^3\#\CP^3\#\CP^3$) we left out above, the associated subspace is generated by $x_1x_2$, $x_1x_3$ and $x_2 x_3$, which is not a regular sequence, since $x_1x_3$ is a zero divisor in $\R[x_1,x_2,x_3]/(x_1x_2)$. Therefore $\CP^3\#\CP^3\#\CP^3$ is not rationally elliptic.

The other nonsingular ternary cubic form we left out, $x^3+y^3+z^3+6 x y z$, is not regular, since  
\[x_2 (x_3^2-x_1x_2)=-x_3(x_1^2-x_2x_3)-x_1(x_2^2-x_1x_3),\]
so $x_3^2-x_1x_2$ is a zero divisor in  $\R[x_1,x_2,x_3]/(x_1^2-x_2x_3,x_2^2-x_1x_3)$

The proof of Theorem~\ref{TheoremDimension6Realb2eq3} will be completed by the following lemma.
\begin{table}[!htbp]\begin{center}\caption{\label{TableTernaryCubicForms}Ternary real cubic forms and associated sequence of homogenous polynomials of degree two}

\begin{tabular}{p{3.5cm}p{5.6cm}l}\toprule
cubic form& sequence  & regular\\
\midrule
$0$& $x_1^2$,\, $x_2^2$,\, $x_3^2$,\, $x_1 x_2$,\, $x_1 x_3$,\, $x_2 x_3$&no \\ \mytableextraspace
$x^3$& $x_2^2$,\, $x_3^2$,\, $x_1 x_2$,\, $x_1 x_3$,\, $x_2 x_3$  &  no \\ \mytableextraspace
$x^2y$ &$x_2^2$,\, $x_1 x_3$,\, $x_2 x_3$,\, $x_3^2$&  no\\ \mytableextraspace
$x^2 y - x y^2$ & $x_1^2+x_1 x_2 + x_2^2$,\, $x_1 x_3$,\, $x_2 x_3$,\, $x_3^2 $ & no\\ \mytableextraspace
$x(x^2+y^2)\sim  x^3+y^3 $ & $x_1x_2$,\, $x_1x_3$,\, $x_2 x_3$,\, $x_3^2$& no\\ \mytableextraspace
$x y z$ & $x_1^2$,\, $x_2^2$,\, $x_3^2$ & yes\\ \mytableextraspace
$z(x^2+y^2)$ & $x_1 x_2$,\, $x_1^2-x_2^2$,\, $x_3^2$ & yes\\ \mytableextraspace
$x(xz -y^2)$ & $x_2^2+x_1 x_3$,\, $x_3^2$,\, $x_2 x_3$ & no\\ \mytableextraspace
{\raggedright $z(x^2+y^2-z^2)$\par $\sim z (3 x^2 + 3 y^2-z^2)$ }&$x_1x_2$,\, $x_1^2+x_3^2$,\, $x_2^2 + x_3^2$  & yes\\ \mytableextraspace
{\raggedright $x(x^2+y^2-z^2)$\par$\sim x(x^2+3 y^2 - 3 z^2)$} & $x_2 x_3$,\, $x_1^2-x_2^2$,\, $x_1^2+x_3^2$ & yes\\ \mytableextraspace
{\raggedright $x(x^2+y^2+z^2)$\par$\sim x(x^2+3 y^2 + 3 z^2)$} &$x_2 x_3$,\, $x_1^2-x_2^2$,\, $x_1^2-x_3^2$&  yes\\ \mytableextraspace
$x^3-3 y^2 z $ & $x_1 x_2$,\, $x_1 x_3$,\, $x_3^2$ & no\\ \mytableextraspace
$x^3+3x^2 z-3 y^2 z$ & $x_1 x_2$,\, $x_3^2$,\, $x_1^2 - x_1 x_3 + x_2^2$ & yes\\ \mytableextraspace
$x^3-3 x^2 z-3y^2 z$ &$x_1 x_2$,\, $x_3^2$,\, $x_1^2 + x_1 x_3 - x_2^2$  & yes\\ \mytableextraspace
$x^3+y^3+z^3+6 \sigma x y z$, $\sigma \neq-\frac{1}{2}$&$\sigma x_1^2-x_2 x_3$,\, $\sigma x_2^2-x_1 x_3$,\, $\sigma x_3^2 - x_1 x_2$&\\ 
\bottomrule
\end{tabular}\end{center}\end{table}

\begin{lem}The subspaces associated to the cubic forms $x y z$, $z(x^2+y^2)$, $z(x^2+y^2-z^2)$, $x(x^2+y^2-z^2)$, $x(x^2+y^2+z^2)$, $x^3+3x^2 z-3 y^2 z$, $x^3-3 x^2 z-3y^2 z$, and $x^3+y^3+z^3+6 \sigma x y z$ for $\sigma \not\in\{0,1,-\frac{1}{2}\}$ are generated by a regular sequence, while the ones associated to $0$, $x^3$, $x^2y$, $x^2 y - x y^2$, $x(x^2+y^2)$, $x(xz -y^2)$,  $x^3-3 y^2 z $  and $x^3+y^3+z^3+6 \sigma x y z$ for $\sigma\in \{0,1\}$ are not generated by a regular sequence.
\end{lem}

\begin{proof}
Bases for the associated subspaces are given in Table~\ref{TableTernaryCubicForms}. The regularity of the sequences associated to $x y z$, $z(x^2+y^2)$, $z(x^2+y^2-z^2)$, $x(x^2+y^2-z^2)$, $x(x^2+y^2+z^2)$, $x^3+3x^2 z-3 y^2 z$, $x^3-3 x^2 z-3y^2 z$, and $x^3+y^3+z^3+6 \sigma x y z$ for $\sigma \not\in\{0,1,-\frac{1}{2}\}$ is seen using the application of Hilbert's Nullstellensatz already used in the discussion following Lemma~\ref{LemmaBinaryCubicFormsRealization}. Except for the ones associated to $x(xz -y^2)$ and $x^3+y^3+z^3+6 \sigma x y z$ with $\sigma \in \{0,1\}$, all non-regular sequences contain two elements of the form $x_i x_j$ and $ x_i x_k$ with $\{i,j,k\}=\{1,2,3\}$. These are non-regular, since $x_ix_j \cdot x_k \in (x_i x_k)$. For $x(xz -y^2)$, the two elements $x_3^2$ and $x_2 x_3$ allow a similar construction and the last case has been treated above.
 \end{proof}

\begin{table}[!htb]\begin{center}\caption{\label{TableCubicFormsExamples}Ternary real cubic forms and examples of manifolds with cohomology ring having the cubic form associated to it}

\begin{tabular}{p{3.2cm}p{5.3cm}p{2.2cm}}\toprule
cubic form &example  & rationally\\
\midrule
$0$& $(\Sph^2 \times \Sph^4) \#(\Sph^2 \times \Sph^4) \#(\Sph^2 \times \Sph^4 )$&hyperbolic \\\mytableextraspace
$x^3$&$\CP^3 \# (\Sph^2 \times \Sph^4) \#(\Sph^2 \times \Sph^4)$&  hyperbolic \\\mytableextraspace
$x^2y$ & $(\CP^2 \times \Sph^2)\# (\Sph^2 \times \Sph^4) $&  hyperbolic\\\mytableextraspace
$x^2 y - x y^2$ & $(\SU(3)/ \T^2)\# (\Sph^2 \times \Sph^4) $& hyperbolic\\\mytableextraspace
$x(x^2+y^2)$ & $\CP^3 \# \CP^3 \#(\Sph^2 \times \Sph^4) $& hyperbolic\\\mytableextraspace
$x y z$ & $\Sph^2 \times\Sph^2 \times\Sph^2 $, $(\CP^2 \# \overline{\CP}^2) \times  \Sph^2 $ & elliptic\\\mytableextraspace
$z(x^2+y^2)$ & $(\CP^2 \# \CP^2) \times  \Sph^2 $ & elliptic\\\mytableextraspace
$x(xz -y^2)$ & $ $ & hyperbolic\\\mytableextraspace
$z(x^2+y^2-z^2)$ & & elliptic\\\mytableextraspace
$x(x^2+y^2-z^2)$ & $B^3_{b_1,c_1,c_2}$ with $c_2\neq 0$, $c_1\neq \frac{b_1c_2}{2}$&  elliptic\\\mytableextraspace
$x(x^2+y^2+z^2)$ & $B^1_{c_1,c_2} $, with $(c_1,c_2)\neq(0,0)$ &   elliptic\\\mytableextraspace
$x^3-3 y^2 z $ & $\CP^3 \# (\CP^2 \times \Sph^2) $ & hyperbolic\\\mytableextraspace
$x^3+3x^2 z-3 y^2 z$ & $ $ &  elliptic \\\mytableextraspace
$x^3-3 x^2 z-3y^2 z$ &$B^2_{0,b_3} $ with $b_3\neq 0$&  elliptic\\\mytableextraspace
$x^3+y^3+z^3+6 \sigma x y z$, $\sigma \neq-\frac{1}{2},0,1$&$B^\text{sp}$&elliptic\\\mytableextraspace
$x^3+y^3+z^3+6 \sigma x y z$, $\sigma \in\{0,1\}$&$\CP^3\#\CP^3\#\CP^3$&hyperbolic\\
\bottomrule
\end{tabular}\end{center}\end{table}

In some of these cases we can give examples of manifolds which have these cubic forms, see Table~\ref{TableCubicFormsExamples}. Most of these are easy to see. We concentrate on the manifolds $B^1_{c_1,c_2}$, $B^2_{a_3,b_3}$ and $B^3_{b_1,c_1,c_2}$. They are certain biquotients that have been studied by DeVito \cite{DeV,DeVitoBiqu67}. They are given as quotients of $\Sph^3\times \Sph^3 \times \Sph^3$ by a $\T^3$-action. The general form of the actions is given by
\begin{align*}
&(u,v,w).((p_1,p_2),(q_1,q_2),(r_1,r_2))\\
&\qquad \qquad=((u p_1,u^{a_1} v^{a_2} w^{a_3} p_2), (u q_1,u^{b_1} v^{b_2} w^{b_3} q_2), (u r_1,u^{c_1} v^{c_2} w^{c_3} r_2)).
\end{align*}
where $(u,v,w)\in \T^3$ and $((p_1,p_2),(q_1,q_2),(r_1,r_2))\in (\Sph^3)^3\subset (\C^2)^3$. The action is determined by the matrix $\left(\begin{smallmatrix}a_1&a_2&a_3\\b_1&b_2&b_3\\c_1&c_2&c_3\\\end{smallmatrix}\right)\in \Z^{3 \times 3}$. The biquotients $B^1_{c_1,c_2}$, $B^2_{a_3,b_3}$ and $B^3_{b_1,c_1,c_2}$ are given by the matrices
\[\begin{pmatrix}1&2&0\\1&1&0\\c_1&c_2&1\\\end{pmatrix},\begin{pmatrix}1&2&a_3\\1&1&b_3\\0&0&1\\\end{pmatrix}\text{ and }\begin{pmatrix}1&0&0\\b_1&1&0\\c_1&c_2&1\\\end{pmatrix}.\]
The manifold $B^\text{sp}$ also is a biquotient of this form, the first of the sporadic examples in \cite{DeVitoBiqu67} with action determined by the matrix \[ \begin{pmatrix}1&2&2\\1&1&2\\1&1&1\\\end{pmatrix}.\] 
Their cohomology rings have been computed in \cite[Proposition 4.9]{DeVitoBiqu67}:{\medmuskip=2.5mu
\begin{align*}
\coh^*(B^1_{c_1,c_2};\Z) &\isom \Z[u,v,w]/(u^2 + 2 u v, v^2+ u v, w^2+ c_1 u w + c_2 v w),\\
\coh^*(B^2_{a_3,b_3};\Z)& \isom \Z[u,v,w]/(u^2 + 2 u v + a_3 u w, v^2+ u v + b_3 v w, w^2),\\ 
\coh^*(B^3_{b_1,c_1,c_2};\Z)& \isom \Z[u,v,w]/(u^2, v^2+b_1 u v, w^2+ c_1 u w + c_2 v w),\\
\coh^*(B^{\text{sp}};\Z) &\isom \Z[u,v,w]/(u^2+2uv+2uw, v^2+ u v+2vw, w^2+uw+vw)
\end{align*}
with $u,v,w$ of degree 2.}

For some of these biquotients we will now compute the cubic form, associated to their cohomology rings.

Consider first $B^1_{c_1,c_2}$ with $(c_1,c_2)\neq (0,0)$. Let $\alpha=\sqrt{c_2^2+(2 c_1-c_2)^2}\neq0$ and $x_1,x_2, x_3$ be the basis of $\coh^2(B^1_{c_1,c_2};\R)$ with $u=-2 x_3$, $v=x_2+x_3$ and $w=-\tfrac{\alpha}{2} x_1 - \tfrac{c_2}{2} x_2 + (c_1-\tfrac{c_2}{2}) x_3$. Then
\begin{align*}u^2 + 2 u v&= -4 x_2 x_3\\
v^2+u v&= - (x_1^2-x_2^2) +(x_1^2-x_3^2)\\
w^2 + c_1 u w+ c_2 v w&= \tfrac{c_2^2}{4} \;(x_1^2-x_2^2) +\tfrac{1}{4} (2c_1-c_2)^2\;(x_1^2-x_3^2)\\ &\qquad+\big(c_1 c_2 - \tfrac{c_2^2}{2}\big)\; x_2 x_3,
\end{align*}
which spans the same subspace of $(\R[x_1,x_2,x_3])^2$ as $ x_2x_3,x_1^2-x_2^2,x_1^2-x_3^2$, the sequence associated to $x(x^2+y^2+z^2)$, see Table~\ref{TableTernaryCubicForms}. Therefore $B^1_{c_1,c_2}$ has $x(x^2+y^2+z^2)$ as associated cubic form.

Next consider $B^2_{0,b_3}$ with $b_3\neq 0$. Let $x_1,x_2,x_3$ be the basis of $\coh^2(B^2_{0,b_3};\R)$ with $u=-\frac{b_3^{1/3} }{2^{2/3}}(2 x_1-2 x_2+x_3)$, $v=-2^{1/3} b_3^{1/3} x_2$ and $w=\frac{1}{2^{2/3} b_3^{2/3}}x_3$. Then
\begin{align*}
u^2 + 2 u v&=\frac{b_3^{2/3}}{2^{4/3}}\; x_3^2+2^{2/3} b_3^{2/3}\;(x_1^2 + x_1 x_3 - x_2^2),
\\v^2+ u v + b_3 v w&=2^{2/3}b_3^{2/3} \;x_1 x_2,
\\w^2&=\left(\tfrac{1}{2^{2/3} b_3^{2/3}}\right)^2 \;x_3^2.
\end{align*}
Consulting Table~\ref{TableTernaryCubicForms} shows that the associated cubic form is $x^3-3 x^2 z-3y^2 z$.

Now consider  $B^3_{b_1,c_1,c_2}$ with  $c_2\neq 0$ and $2c_1 \neq b_1 c_2$. Let $x_1,x_2,x_3$ be the basis of $\coh^2(B^3_{b_1,c_1,c_2};\R)$ with $u=c_2(x_2-x_3)$, $v=(c_1-b_1c_2)x_2 + c_1 x_3$ and $w=\tfrac{1}{2} c_2 (b_1c_2-2 c_1)(x_1+x_2)$. Then
\begin{align*}
u^2 &= -c_2^2 (2 f_1+f_2 - f_3)\\
v^2+u v&= (2c_1^2-2b_1 c_1 c_2 +b_1^2 c_2^2)f_1 + (c_1^2-b_1c_1c_2) (-f_2+f_3)\\
w^2 + c_1 u w+ c_2 v w&= \tfrac{1}{4}c_2^2 (b_1c_2-2c_1)^2 f_2,
\end{align*}
with $f_1=x_2 x_3$, $f_2=x_1^2-x_2^3$ and $f_3=x_1^2+x_3^2$. It follows that $B^3_{b_1,c_1,c_2}$ realizes the cubic form $x(x^2+y^2-z^2)$ by again consulting Table~\ref{TableTernaryCubicForms}.

In  $\coh^*(B^\text{sp})$, we have that $u^2v=-2 u v w$, $u^2 w= 0$, $uv^2=0$, $uw^2=-u v w$, $v^2 w=-uvw$, $vw^2=0$, $u^3=4uvw$, $v^3=2uvw$ and $w^3=uvw$. Hence,  the cubic form associated to the cohomology ring of $B^\text{sp}$ is 
\[4x^3+2y^3+z^3-6x^2y-3x z^2-3y^2 z+6xyz.\]
Computing the gradient, it is easy to see, that this form is nonsingular. So for some $\sigma$ it is equivalent to the form $x^3+y^3+z^3+6 \sigma x y z$. A numerical computation shows, that  $\sigma \approx 0.27788$ for $B^\text{sp}$.

\begin{rem}[The complex case]
 The normal forms of complex ternary cubic forms can be found for example in \cite[Section~I.7]{Kraft} or \cite[Section~7.3]{PlaneAlgCurves}. In particular every nonsingular cubic form can be brought to the Hesse normal form $C_\lambda=x^3+y^3+z^3+\lambda x y z$ with $\lambda \in \C$, $\lambda^3\neq 27$. For a given $\lambda $ there are only finitely many $\lambda'$ such that $C_\lambda$ and $C_{\lambda'}$ are equivalent, see \cite[Section~7.3, Theorem~10]{PlaneAlgCurves}. Therefore the results can easily be adapted to the complex case, in particular Proposition~\ref{prop:InfinitelymanyrealHomotopytypesDimension6} still holds for complex homotopy types.
\end{rem}

\section{Seven-dimensional manifolds}\label{sec:DimensionSeven}

As in the six-dimensional case we start with the computation of the possible exponents using the results of Friedlander and Halperin given in Section \ref{susubsec:Exponents}.
\begin{lem}\label{Lemma7dimensionalExponents}
A closed, simply connected, rationally elliptic 7--manifold has one of the following exponents:
\begin{multicols}{2}
\begin{enumerate}
\item[(7.1)]$a=(~)$, $b=(4)$
\item[(7.2)]$a=(1)$, $b=(2,3)$
\item[(7.3)]$a=(2)$, $b=(2,4)$
\item[(7.4)]$a=(1,1)$, $b=(2,2,2)$
\end{enumerate}
\end{multicols}
\end{lem}
Again, most exponents allow only finitely many rational homotopy types.
\begin{lem}\label{Lemma7dimensionalFinitelyManyExamples}
A closed, simply connected, rationally elliptic 7--manifold with exponents like in
\begin{itemize}
\item (7.1) is rationally homotopy equivalent to $\Sph^7$;
\item (7.2) is rationally homotopy equivalent to $\Sph^2 \times \Sph^5$ or $\CP^2 \times \Sph^3$;
\item (7.3) is rationally homotopy equivalent to $\Sph^3\times\Sph^4$.
\end{itemize}
\end{lem}
\begin{proof}
Cases (7.1) and (7.3) are easy. In case (7.2) there  are generators $x\in V^2$, $y_3\in V^3$ and $y_5\in V^5$. For the differential there are three possibilities: $d_1x=0$, $d_1y_3=x^2$ and $d_1y_5=0$ which gives the minimal model of $\Sph^2 \times \Sph^5$, $d_2x=0$, $d_2y_3=0$ and $d_2y_5=x^3$ which gives the minimal model of $\Sph^3\times \CP^2$ and $d_3x=0$, $d_3y_3=x^2$ and $d_3y_5=x^3$. The last model is isomorphic to the first via $\varphi: (\Lambda(x,y_1,y_2),d_3)\to (\Lambda V,d_1)$ with $\varphi(x)=x$, $\varphi(y_3)=y_3$ and $\varphi(y_5)=y_5-x y_3$.
\end{proof}

So we are  left with manifolds having exponents like in case (7.4). First note, that for a minimal Sullivan algebra $(\Lambda V,d)$ with exponents like in (7.4), so that $\dim V^2=2$, $\dim V^3=3$ and $\dim V^i=0$ else, the rank of $d|_{V^3}$ has to satisfy $\rank d|_{V^3}\geq 2$ if $\dim \coh^*(\Lambda V,d)<\infty$.

Consider the minimal Sullivan algebras 
\[(\Lambda V,d_{\tilde{\sigma}})=(\Lambda(x_1,x_2,y_1,y_2,y_3),d_{\tilde{\sigma}})\]
with $\tilde{\sigma}\in \Q^*$, $|x_i|=2$, $|y_j|=3$ and differential given by $d_{\tilde{\sigma}} x_i=0=d_{\tilde{\sigma}}y_3$, $d_{\tilde{\sigma}}y_1= x_1x_2$ and $d_{\tilde{\sigma}}y_2=x_1^2-\tilde{\sigma}x_2^2$.

\begin{lem}\label{LemmaIsomorphicSigmaModels}
Two such models $(\Lambda V, d_{\tilde{\sigma}})$ and $(\Lambda V, d_{\tilde{\sigma}'})$ are isomorphic if and only if the equivalence classes $[\tilde{\sigma}]$ and $[\tilde{\sigma}']$ in $\Q^*/(\Q^*)^2$ agree. 

Let $\sigma=[\tilde{\sigma}]\in \Q^*/(\Q^*)^2$. Then $(\Lambda V, d_{\tilde{\sigma}})$ is the minimal model of a 7--manifold $M^7_\sigma$.
\end{lem}

\begin{proof}
To see that $(\Lambda V, d_{\tilde{\sigma}})$ is the minimal model of a 7--manifold first note that $(\Lambda V, d_{\tilde{\sigma}})\isom (\Lambda(x_1,x_2,y_1,y_2),d_{\tilde{\sigma}})\tensor (\Lambda(y_3),0)$. A short computation shows, that $x_1^2-\tilde{\sigma} x_2^2$, $x_1x_2$ is a regular sequence. So $\coh^*(\Lambda (x_1,x_2,y_1,y_2),d_{\tilde{\sigma}})$ is finite dimensional and $(\Lambda (x_1,x_2,y_1,y_2), d_{\tilde{\sigma}})$ is rationally elliptic. By a theorem of Halperin \cite[Theorem 3]{Halperin77},  $\coh^*(\Lambda (x_1,x_2,y_1,y_2),d_{\tilde{\sigma}})$ and therefore  $\coh^*(\Lambda V,d_{\tilde{\sigma}})$ satisfy Poincaré duality, and by work of Sullivan $(\Lambda V,d_{\tilde{\sigma}})$ is the minimal model of a closed, simply connected 7--manifold. 

Since $\coh^4(\Lambda V,d_{\tilde{\sigma}})$ is one-dimensional, we can identify it with $\Q$ and get a symmetric bilinear form on $\coh^2(\Lambda V,d_{\tilde{\sigma}})$. The determinant of this form is $\tilde{\sigma}$ if we choose $x_1^2$ as a generator of $\coh^4(\Lambda V,d_{\tilde{\sigma}})$ and its equivalence class in $\Q^*/(\Q^*)^2$ is an invariant of the cohomology ring.

If, on the other hand, $\tilde{\sigma},\tilde{\sigma}'\in \Q^*$ with $[\tilde{\sigma}]=[\tilde{\sigma}']$ in $\Q^*/(\Q^*)^2$ are given, then $\sqrt{\tilde{\sigma}'/\tilde{\sigma}}\in \Q$ and $\varphi: (\Lambda V, d_{\tilde{\sigma}})\to (\Lambda V, d_{\tilde{\sigma}'})$ defined by $\varphi(x_1)=x_1$, $\varphi(x_2)= \sqrt{\tilde{\sigma}'/\tilde{\sigma}} \; x_2$, $\varphi(y_1)=\sqrt{\tilde{\sigma}'/\tilde{\sigma}}\; y_1$  and $\varphi(y_j)=y_j$ for $j=2,3$ is an isomorphism.
\end{proof}
\begin{rem}
One can choose \[M^7_{[1]}=(\CP^2\#\CP^2)\times \Sph^3\] and \[M^7_{[-1]}=(\CP^2\#\overline{\CP}^2)\times \Sph^3\simeq_\Q\Sph^2\times \Sph^2\times \Sph^3.\] Here $\overline{\CP}^2$ denotes reversing the orientation and $\simeq_\Q$ denotes being rationally homotopy equivalent.
\end{rem}

\begin{rem}\label{RemarkDefinitionXsigma}
The minimal Sullivan algebras
\[(\Lambda(x_1,x_2,y_1,y_2),d_{\tilde{\sigma}})\]
used in the proof define rationally elliptic spaces $X_\sigma$, $\sigma \in \Q^*/(\Q^*)^2$, of formal dimension 4. These can be realized as four-dimensional orbifolds of nonnegative curvature, see \cite{GGKRW14}. However, $X_\sigma $ is not rationally homotpy equivalent to a manifold, since the intersection form cannot be induced by a unimodular form defined over the free part of the integer cohomology. The proof also shows that $M^7_\sigma\simeq_\Q X_\sigma\times\Sph^3$.
\end{rem}

The last minimal model we need to consider is 
\[(\Lambda V,d)=(\Lambda(x_y,x_2,y_1,y_2,y_3),d), \quad |x_i|=2,|y_j|=3\]
with $dx_i=0$, $dy_1=x_1^2$, $dy_2=x_2^2$ and $dy_3=x_1x_2$. In \cite[Example 2.91]{FOT} it is introduced as the minimal model of an $\Sph^3$-bundle over $\Sph^2 \times \Sph^2$. We will give a description of it as a homogeneous space. Let 
\[K=\left\{\left(\left(\begin{smallmatrix}z&0\\0&z^{-1}\end{smallmatrix}\right),\left(\begin{smallmatrix}w&0\\0&w^{-1}\end{smallmatrix}\right),\left(\begin{smallmatrix}zw&0\\0&(zw)^{-1}\end{smallmatrix}\right)\right)\middle| z,w \in \Sph^1\right\}\leq G:=(\SU(2))^3\]
and $N^7=G/K$. Then, see \cite[Theorem 2.71]{FOT}, a model for $N^7$ is given by $(\Lambda W \oplus \Lambda( sU),d)$, where $\Lambda W=\coh^*(\mathrm{B}K;\Q)$, $\Lambda U=\coh^*(\mathrm{B}G;\Q)$, and $sU$ denotes a shift in degree, so $|su|=|u|-1$ for $u\in U$. The differential is given by $dw=0$ for $w\in W$ and $d(su)=\coh^*(\mathrm{B}\iota)(u)$ for $u \in U$ and $\iota: K\hookrightarrow G$ the inclusion. In our situation, $\Lambda W=\Lambda (x_1,x_2)$ with $|x_i|=2$, $\Lambda (sU)=\Lambda(y_1,y_2,y_3)$ with $|y_j|=3$. The map $\coh^*(\mathrm{B}\iota)$ can be computed from the inclusion of $H$ in the standard maximal torus of $G$. One gets $dy_1=x_1^2$, $dy_2=x_2^2$ and $dy_3=(x_1+x_2)^2$, so the minimal model of $N^7$ is isomorphic to $(\Lambda V, d)$ as above.

\begin{proof}[Proof of Theorem~\ref{TheoremDimension7Rational}]
By Lemma~\ref{Lemma7dimensionalFinitelyManyExamples} we only need to show that a minimal model with exponents like in (7.4) is isomorphic to the minimal model of $N^7$ or some $M^7_\sigma$. Let $(\Lambda V,d)$ be a minimal model with exponents like in (7.4). Then as we already noted $\rank d|_{V^3}\geq2$.

Suppose $\rank d|_{V^3}=2$. Then $\coh^4(\Lambda V,d)$ is one-dimensional and the multiplication $\coh^2(\Lambda V,d)\times\coh^2(\Lambda V,d)\to \coh^4(\Lambda V,d)$ can be interpreted as a symmetric bilinear form. Choose a basis $x_1, x_2$ of $V^2=\coh^2(\Lambda V,d)$ that diagonalizes this form. Then $x_1 x_2\in (\Lambda V)^4$ is exact, so there exists $y_1 \in V^3$ with $dy_1=x_1x_2$. Choose $y_3 \in\ker d|_{V^3}$. Then choose $y_2\in V^3$ such that $y_1,y_2,y_3$ is a basis. By subtracting a multiple of $y_1$, scaling and possibly interchanging $x_1$ and $x_2$, we can assume that $dy_2=x_1^2+a x_2^2$ for some $a\in \Q$. If $a=0$ then for every $n\in \N$, we had that $x_2^n$ is closed but not exact, so $a\neq0$. 

If $\rank d|_{V^3}=3$, then the minimal model is obviously the one of $N^7$.\end{proof}

Using the classification of rationally elliptic manifolds in lower dimensions, the classification of compact, simply connected homogeneous manifolds in dimensions up to 9 by Klaus \cite{Klaus} and low-dimensional cohomogeneity one manifolds by Hoelscher (\cite{Hoel10class}  and \cite{Hoel10hom}) one can prove the following.
\begin{samepage}
\begin{prop}\label{PropM7sigmanot} For $\sigma \in \Q^*/(\Q^*)^2\setminus\{ [1],[-1]\}$ the manifold $M_\sigma^7$ does not have the rational homotopy type of 
\begin{enumerate}[label=\alph*)]
\item a product of closed, simply connected manifolds,
\item a bundle over a closed, simply connected, rationally elliptic manifold of dimension $\leq 5$ with fibre a closed, simply connected manifold,
\item a closed, simply connected, homogeneous space,
\item a closed, simply connected cohomogeneity one manifold.
\end{enumerate}
\end{prop}
\end{samepage}

The classification of real homotopy types of closed, simply connected, rationally elliptic 7--manifolds now reduces to understanding which of the rational homotopy types of Theorem~\ref{TheoremDimension7Rational} give the same real one. Lemma~\ref{LemmaIsomorphicSigmaModels} carries over to the real case, replacing $\Q^*/(\Q^*)^2$ by $\R^*/(\R^*)^2=\{1,-1\}$. Since $M^7_{[1]}=(\CP^2\#\CP^2)\times \Sph^3$, $M^7_{[-1]}=(\CP^2\#\overline{\CP}^2)\times \Sph^3$, and the other manifolds in Theorem~\ref{TheoremDimension7Rational} already differ by their Betti numbers, we get the following proposition.
\begin{prop}
A closed, simply connected 7--manifold is rationally  elliptic if and only if  it has the real homotopy type of one of the following manifolds: \\
$\Sph^7$, $\Sph^2 \times \Sph^5$, $\CP^2 \times \Sph^3$, $\Sph^3 \times \Sph^4$, $N^7$, $(\CP^2\#\CP^2)\times \Sph^3$ or $(\CP^2\#\overline{\CP}^2)\times \Sph^3$.
\end{prop}
Of these manifolds the only ones having the same complex homotopy type are  $(\CP^2\#\CP^2)\times \Sph^3$ and $(\CP^2\#\overline{\CP}^2)\times \Sph^3$. Since $ \CP^2\#\overline{\CP}^2\simeq_\Q \Sph^2\times \Sph^2$ this shows the following for the complex homotopy types.

\begin{prop}
A closed, simply connected 7--manifold is rationally  elliptic if and only if  it has the complex homotopy type of one of the following manifolds: \\
$\Sph^7$, $\Sph^2 \times \Sph^5$, $\CP^2 \times \Sph^3$, $\Sph^3 \times \Sph^4$, $N^7$ or $\Sph^2 \times \Sph^2 \times \Sph^3$.
\end{prop}

\section{Higher dimensions}\label{sec:DimensionsEightAndNine}
\subsection{Dimension 8}
As before, we start by computing the possible exponents of closed, simply connected, rationally elliptic 8--manifolds using the results of Friedlander and Halperin mentioned in Section \ref{susubsec:Exponents}..
\begin{lem}\label{Lemma8dimensionalExponents}
A closed, simply connected, rationally elliptic 8--manifold has one of the following exponents:
\begin{multicols}{2}
\begin{enumerate}
\item[(8.1)]$a=(~)$, $b=(2,3)$
\item[(8.2)]$a=(1)$, $b=(5)$
\item[(8.3)]$a=(2)$, $b=(6)$
\item[(8.4)]$a=(4)$, $b=(8)$
\item[(8.5)]$a=(1)$, $b=(2,2,2)$
\item[(8.6)]$a=(1,1)$, $b=(2,4)$
\item[(8.7)]$a=(1,1)$, $b=(3,3)$
\item[(8.8)]$a=(1,2)$, $b=(3,4)$
\item[(8.9)]$a=(1,3)$, $b=(2,6)$
\item[(8.10)]$a=(2,2)$, $b=(4,4)$
\item[(8.11)]$a=(1,1,1)$, $b=(2,2,3)$
\item[(8.12)]$a=(1,1,2)$, $b=(2,2,4)$
\item[(8.13)]$a=(1,1,1,1)$, $b=(2,2,2,2)$
\end{enumerate}
\end{multicols}
\end{lem}
In eight of these cases we show that there are only finitely many possible rational homotopy types with the given exponents.
\begin{prop}
In cases $(8.1)$, $(8.2)$, $(8.3)$, $(8.4)$, $(8.5)$, $(8.8)$, $(8.9)$ and $(8.10)$ of Lemma~\ref{Lemma8dimensionalExponents} there are only finitely many rational homotopy types of closed, simply connected 8--manifolds with these exponents. They are:
\begin{multicols}{3}
\begin{enumerate}
\item[(8.1)] $\Sph^3 \times \Sph^5$
\item[(8.2)] $\CP^4$
\item[(8.3)] $\mathbb{HP}^2 $
\item[(8.4)] $\Sph^8$
\item[(8.5)] $\Sph^2 \times \Sph^3 \times \Sph^3$
\item[(8.8)] $\CP^2 \times \Sph^4$
\item[(8.9)] $\Sph^2 \times \Sph^6$
\item[(8.10)] $\Sph^4 \times \Sph^4$,\\ $\HP^2\#\HP^2$
\end{enumerate}
\end{multicols}
\end{prop}
\begin{proof} 
Most is easy, so we concentrate on (8.10).
Let $M$ be a manifold with exponents like in (8.10). Then there is a basis $\omega_1,\omega_2$ of $\coh^4(M;\Q)$ such that $\omega_1 \omega_2=0$ and $\omega_1^2=\varepsilon \omega_2^2$, $\varepsilon=\pm1$. Choose $x_1,x_2\in V^4$ corresponding to $\omega_1,\omega_2$. Then there are $y_1,y_2\in V^7$ with $dy_1=x_1x_2$ and $dy_2=x_1^2-\varepsilon x_2^2$. For $\varepsilon =1$ this is the minimal model of $\HP^2 \# \HP^2$, for $\varepsilon=-1$ it is isomorphic to the one of $\Sph^4 \times \Sph^4$.
\end{proof}
\begin{rem}
In case (8.10) there is  an infinite family of simply connected rationally elliptic spaces that are not rationally homotopy equivalent to a manifold, analogous to the four-dimensional family $X_\sigma$.
\end{rem}

\begin{prop}
The rational homotopy types of closed, simply connected, rationally elliptic 8--manifolds with exponents like in case (8.12) of Lemma~\ref{Lemma8dimensionalExponents} are exactly the ones given by the  $X_\sigma \times \Sph^4$ with $\sigma \in \Q^* /(\Q^*)^2$. In particular, there are infinitely many of these.
\end{prop}

\begin{proof}
Let $(\Lambda V,d)$ be the minimal model of such an 8--manifold.  Then $\dim V^2= \dim V^3=2$, $\dim V^4=\dim V^7=1$ and $\dim V^k=0$ else. Then $d(V^2)=\{0\}$ and $d(V^3)\subset \Lambda^2 V^2$, because of the minimality of the model. 

Suppose $\rank d|_{V^3}\neq 2$. If $\rank d|_{V^3}=1$, let $0\neq y\in V^3$ with $dy=0$. Let $0\neq a\in V^4$. Then $da= y v$ for some $v \in V^2$, so $d(y a)=0$. But $ya\in (\Lambda V)^7 $ is not exact, since $d((\Lambda V)^6)\subset \Lambda^2 V^2 \cdot V^3$. So we have  $\coh^7(\Lambda V,d))\neq \{0\}$, a contradiction. If $\rank d|_{V^3}=0$, then \[\dim \Kern d|_{(\Lambda V)^{10}}\geq\dim(\Lambda^5 V^2\oplus(\Lambda^2 V^2) \cdot (\Lambda^2 V^3))=9\] and \[\rank(d|_{(\Lambda V)^9})\leq \dim( V^2 \cdot V^3 \cdot V^4\oplus V^2 \cdot V^7)=6,\] so $\coh^{10}(\Lambda V,d)\neq\{0\}$, a contradiction. 

Therefore $\rank d|_{V^3}= 2$, so we can choose bases $x_1,x_2$ of $V^2$ and $y_1,y_2$ of $V^3$ such that $dy_1=x_1^2-\tilde{\sigma} x_2^2$ for some $\tilde{\sigma} \in \Q$ and $dy_2=x_1 x_2$. Furthermore let $0\neq a \in V^4$. Then $da=0$, since there are no closed elements in $(\Lambda V)^5$. 

Suppose now that $\tilde{\sigma}=0$. Then $x_2^n$ or $a^n$ is closed, but not exact for every $n$, a contradiction. So $\tilde{\sigma}\neq 0$. Now the only non-exact, closed elements of $(\Lambda V)^8$ are multiples of $a^2$, so up to isomorphism, a generator $z\in V^7$ satisfies $dz=a^2$, which gives the minimal model of $X_\sigma\times \Sph^4$ for $\sigma=[\tilde{\sigma}]$. 

Since their cohomology rings are pairwise non-isomorphic, the $X_\sigma \times \Sph^4$, $\sigma \in \Q^*/(\Q^*)^2$,  have different homotopy types. Since their intersection form is given by $x^2-y^2$, they can be realized as a manifold by Sullivan's realization result, see Section~\ref{subsubsec:RealizationManifold}.
\end{proof}


\subsection{Dimension 9}

Again we compute the possible exponents of closed, simply connected, rationally elliptic 9--manifolds using the results of Friedlander and Halperin mentioned in Section \ref{susubsec:Exponents} and show that in seven of the nine cases there are only finitely many rational homotopy types with the given exponents.
\begin{lem}\label{Lemma9dimensionalExponents}
A closed, simply connected, rationally elliptic 9--manifold has one of the following exponents:
\begin{multicols}{2}
\begin{enumerate}
\item[(9.1)]$a=()$, $b=(5)$
\item[(9.2)]$a=()$, $b=(2,2,2)$ 
\item[(9.3)]$a=(1)$, $b=(2,4)$ 
\item[(9.4)]$a=(1) $, $b=(3,3)$
\item[(9.5)]$a=(2)$, $b=(3,4)$ 
\item[(9.6)]$a=(3)$, $b=(2,6)$ 
\item[(9.7)]$a=(1,1)$, $b=(2,2,3)$ 
\item[(9.8)] $a=(1,2)$, $b=(2,2,4)$ 
\item[(9.9)] $a=(1,1,1)$, $b=(2,2,2,2)$ 
\end{enumerate}
\end{multicols}
\end{lem}

\begin{prop}\label{ProprationaleindeutigDim9}
In cases (9.1)---(9.6) and (9.8) of Lemma~\ref{Lemma9dimensionalExponents} there are only finitely many rational homotopy types of closed, simply connected 9--manifolds with these exponents. They are:
\begin{multicols}{2}
\begin{enumerate}
\item[(9.1)]$\Sph^9$
\item[(9.2)]$\Sph^3 \times \Sph^3 \times \Sph^3$
\item[(9.3)]$\Sph^2 \times \Sph^7$, $\Sph^3 \times \CP^3$
\item[(9.4)]$\Sph^5 \times \CP^2$
\item[(9.5)]$\Sph^4 \times \Sph^5$
\item[(9.6)]$\Sph^3 \times \Sph^6$
\item[(9.8)] $\Sph^2 \times \Sph^3 \times \Sph^4$ 
\end{enumerate}
\end{multicols}
\end{prop}

Let $E=\gamma\oplus\varepsilon$ be the complex rank 2 vector bundle over $\CP^3 \# \CP^3$ which is obtained as the sum of a trivial line bundle $\varepsilon$ and the line bundle $\gamma$ with first Chern class $-(x_1+x_2)$ for generators $x_1$, $x_2$ of $\coh^2(\CP^3\#\CP^3)$ coming from the two $\CP^3$ summands. Let $M^8=P(E)$ be the projectified bundle. By the Leray-Hirsch theorem, the cohomology ring of $M^8$ is given by
\[\coh^*(M^8;\Q)\isom \Q[x_1,x_2,y]/(x_1x_2,x_1^3-x_2^3,y^2-x_1 y- x_2 y),\]
where $y$ is of degree 2.

Let $N^9$ be the principal circle bundle over $M^8$ with first Chern  class given by $y - 2 x_1$. Using the Serre spectral sequence, we can compute the cohomology ring of $N^9$. We get that  $\coh^{\leq4}(N^9;\Q)$ is generated by $x_1$ and $x_2$ with relations $x_1x_2=0=x_1^2$. 

From the construction it is clear that $N^9$ is rationally elliptic. Since the second Betti number $\betti_2(N^9)=2$, by Lemma~\ref{Lemma9dimensionalExponents} the exponents of $N^9$ are like in case (9.7).

\begin{samepage}
\begin{prop}
A closed, simply connected 9--manifold with exponents like in (9.7) of Lemma~\ref{Lemma9dimensionalExponents} has the rational homotopy type of $N^9$, $X_\sigma \times \Sph^5$ (see Remark~\ref{RemarkDefinitionXsigma}) for some $\sigma \in  \Q^* / (\Q^*)^2$ or $M^6 \times \Sph^3$ for a closed, simply connected, rationally elliptic 6--manifold $M^6$ with $\betti_2(M^6)=2$.
\end{prop}
\end{samepage}

\begin{proof}
Let  $(\Lambda V,d)$ be the minimal model of such a 9--manifold $M$. In particular, $\dim V^2=\dim V^3=2$ and $\dim V^5=1$.  If $\ker(d|_{V^3})\neq\{0\}$, then $M\simeq_\Q X\times \Sph^3$, where $X$ is of formal dimension 6. Since $X$ is rationally elliptic, $X\simeq_\Q M^6$, with $M^6$ like in the statement of the proposition.

If $\ker(d|_{V^3})=\{0\}$, then $\dim\coh^4(\Lambda V,d)=1$. We can then choose bases $x_1$, $x_2$ of $V^2$ and $y_1$, $y_2$ of $V^3$ such that $dy_1=x_1 x_2$ and $dy_2= x_1^2+ a x_2^2$ for some $a\in \Q$. If $a \neq 0$, then $(\Lambda V, d)$ is isomorphic to the minimal model of $X_\sigma \times \Sph^5$ with $\sigma$ the equivalence class of $a$ in $\Q/(\Q^*)^2$.

Suppose now $a=0$. Then, up to isomorphism, we can choose $0\neq z \in V^5$ with $d z = x_2^3$. Therefore $\coh^{\leq 4}(\Lambda V,d) \isom \coh^{\leq 4}(N^9)$. Since $N^9$ has the right exponents and the cohomology ring of $N^9$ is non-isomorphic to all of the previously calculated, $(\Lambda V,d)$ is the minimal model of $N^9$. 
\end{proof}

In the remaining case (9.9) of Lemma~\ref{Lemma9dimensionalExponents} there are products $M_\sigma \times \Sph^2$ and $N^7 \times \Sph^2$ of seven-dimensional manifolds with $\Sph^2$ and products of $\Sph^3$ with closed, simply connected, rationally elliptic 6--manifolds with $\betti_2=3$. But there are also examples not having the rational homotopy type of a product.

As an example of such a manifold consider the principal $\Sph^1$-bundle $Y$ over $\Sph^2\times\Sph^2\times\Sph^2\times\Sph^2$ with first Chern class $c_1(Y)=x_1+x_2+x_3+x_4$, where the $x_i$ are generators of the integral cohomology rings of the $\Sph^2$ factors. Using the Serre spectral sequence one can compute the cohomology ring of $Y$. In particular,  $\coh^2(Y;\Q)$ is generated by $[x_1]$, $[x_2]$ and $[x_3]$. The products of these generate $\coh^4(Y;\Q)$ subject to relations $[x_i]^2=0=[x_1][x_2]+[x_1][x_3]+[x_2][x_3]$. Now suppose $Y$ is rationally homotopy equivalent to a product. Due to the classification in dimensions 5 and below, it then has the rational homotopy type of  a product with $\Sph^2$, $\Sph^3 $ or $\Sph^5$. A product with $\Sph^5$ is not possible, since $\betti_2(Y)=3$ and $\betti_2(X)\leq 2$ for a simply connected, rationally elliptic space $X$ of formal dimension 4. As $\betti_3(Y)=0$, we can also exclude a product with $\Sph^3$. By our classification in dimension 7, the last case is that of a product $M_\sigma \times \Sph^2$ or $N^7 \times \Sph^2$. To exclude this, consider the set of elements of the respective second complex cohomology group with vanishing square. For $M_\sigma^7\times \Sph^2$ this is the union of three one-dimensional subspaces, for $N^7\times \Sph^2$ it is the union of a one and a two-dimensional subspace, while for $Y$ it is the union of the four one-dimensional subspaces generated by $[x_1]$, $[x_2]$, $[x_3]$ and $[x_1]+[x_2]+[x_3]$, respectively.

The same argument holds for the family of 9-dimensional biquotients considered by Totaro \cite{Tot}, giving rise to infinitely many rational homotopy types of simply connected, rationally elliptic  9-manifolds with exponents like in (9.9), that do not have the rational homotopy type of a product.

\bibliographystyle{alpha}
\bibliography{refs}{}
\end{document}